\documentclass{amsart}
\usepackage{latexsym,amssymb}
\usepackage{graphicx}
\usepackage{verbatim, color, multicol, lipsum}
\theoremstyle{plain}
\newtheorem{theorem}{Theorem}[section]
\newtheorem{corollary}[theorem]{Corollary}
\newtheorem{lemma}[theorem]{Lemma}
\newtheorem{proposition}[theorem]{Proposition}
\theoremstyle{problem}

\theoremstyle{definition}
\newtheorem{definition}[theorem]{Definition}
\theoremstyle{remark}
\newtheorem{remark}[theorem]{Remark}
\numberwithin{equation}{section}

 \newcounter{thlistctr}
 \newenvironment{thlist}{\
 \begin{list}%
 {\alph{thlistctr}}%
 {\setlength{\labelwidth}{2ex}%
 \setlength{\labelsep}{1ex}%
 \setlength{\leftmargin}{6ex}%
 \usecounter{thlistctr}}}%
 {\end{list}}

\begin{document}
\title[Bases for Boolean algebras] {Johnson's Axioms revisited:  Bases for \\Boolean algebras 
containing identities of associative type. I}
\dedicatory{ Dedicated to the memories of my 
daughter Geeta and my son Sanjay} 
\author{Hanamantagouda P. Sankappanavar}
\address{Department of Mathematics\\
         State University of New York\\
         New Paltz, New York 12561}
\email{sankapph@hawkmail.newpaltz.edu}

\begin{abstract}
This paper is inspired by Johnson's 1892 paper 
where he has given an axiomatization for the variety of Boolean algebras (equivalently, for classical propositional calculus).  The fact that Johnson's axioms include the associative law, the most well-known identity of associative type of length 3, led us naturally to  the question as to whether there are axiom systems for Boolean algebras that include an identity of associative type, other than the associative law.  It turns out that the answer to this question is positive in the strongest sense possible.  In fact, corresponding to each of the 14 non-equivalent identities of associative type, there is a base containing that identity.
Our goal, in this sequence of papers, of which the present paper is the first, is to describe 
 bases for Boolean algebras, 
each of which contains at least one identity of associative type of length 3.
In this paper, we prove, firstly, that Johnson's third axiom is redundant.  Secondly, 
we give several equational bases--some 2-bases and some 3-bases--for the variety of Boolean algebras, each containing an identity of associative type, other than the associative law.   Each of these bases can be easily converted into an axiomatization for the classical propositional logic, as well.    
\end{abstract}

\maketitle

\section{Introduction}

 This paper is the first in a series of  papers inspired by W. E. Johnson's 1892 paper \cite{Jo92c} where he has given a 5-base (i.e. an equational axiomatization consisting of exactly five independent axioms), in the language $\langle \land, \neg \rangle$, for the variety of Boolean algebras (equivalently, for the classical propositional calculus).  The fact that Johnson's axioms include the associative law, the most well-known identity of associative type of length 3, led us naturally to  the question as to whether there are axiom systems for Boolean algebras that include an identity of associative type, other than the associative law.  It turns out that the answer to this question is positive in the strongest sense possible.  In fact, interestingly enough, corresponding to each of the 14 non-equivalent identities of associative type, we have found a base containing that identity.  Of these bases, some are 2-bases (i.e., having exactly two independent identities) and others are 3-bases (having exactly three identities).  We wish to present these findings in this sequence of  papers.  In this paper, we first prove that one of Johnson's axioms, namely the third axiom, is redundant.  
Secondly, as our main goal, we present in this sequence of papers, through a systematic and exhaustive search, several equational bases for Boolean algebras, each of which contains at least one of the identities (Ai), $i=1, 2, \cdots, 14$, of associative type of length 3, thus giving a strongest positive answer to the question raised above.  More precisely, we present 
four kinds 
of (equational) bases for Boolean algebras such that 
  each base of the $j$th kind, $j=1,2,3,$, contains (exactly) $j$ identities from \{(Ai), $i=1, 2, \cdots, 14\}$.
 Some of the (Ai)'s 
 lead to 2-bases, some others to 3-bases, and still others to 4-bases.  For example, each of (A4), A5), (A6), (A8), (A13) and (A18), together with one more identity-- a variation of Johnson's fifth identity--yields a 2-base for the variety of Boolean algebras.  In this paper, we present 2-bases containing (A6), (A8), (A5), (A13) and 3-bases containing (A9) and (A1), while others will appear in the sequel to this paper.

\medskip
\section{Johnson's Axioms Revisited: A Simplification of Johnson's Axioms} 

In this section, we show that Johnson's third axiom is redundant.
 Johnson wrote in \cite{Jo92c} the following (and we quote):

\begin{quote}

 ``The following are the formal universal laws of propositional
 synthesis, expressed by means of =, the symbol of equivalence:
 \begin{itemize}
 \item[I] The Commutative Law: \ $xy = yx,$
 \item[II] The Associative Law: \quad \ \ $xy.z = x.yz$,
\item[III] The Law of Tautology: \ \ \ \ $xx = x, $ 
\item[IV] The Law of Reciprocity: \ \ $\overline{\overline{x}} = x,$   
\item[V] The Law of Dichotomy: \ \ \  $\overline{x} = \overline{xy}\ \overline{x\overline{y}}$.'' \\ 
\end{itemize}
\end{quote}

From the above axioms it is clear that Johnson's connectives are conjunction and negation, represented by juxtaposition and a superimposed bar, respectively.
For the sake of convenience, we will recast the Johnson's axioms given above into an algebraic form (in modern terminology and notation) in the following theorem:

\begin{theorem}{\rm(}Johnson \cite{Jo92c}{\rm)} An algebra ${\bf B} = \langle B, \land, ' \rangle$ is a
    {\it Boolean algebra} if and only if the following five axioms hold in ${\bf B}$:
\begin{enumerate}
   \item[{\rm(J1)}] $x \land y \approx y \land x,$
      \item[{\rm(J2)}] $(x \land y) \land z \approx x \land (y \land z)$,
   \item[{\rm(J3)}] ${x \land x  \approx x},$  
   \item[{\rm(J4)}]  ${x'' \approx x},$
   \item[{\rm(J5)}] ${x' \approx (x \land y)' \land (x \land y')'}$.
\end{enumerate}
\end{theorem}

Let $\mathbb{BA}$ denote the variety of Boolean algebras.  Recall that a {\bf base} for the variety $\mathbb{BA}$ is an independent set of identities that defines the variety $\mathbb{BA}$. Let $n \in \mathbb{N}$.  An $n$-base is a base of size $n$.


We will now show that the Johnson's axiom $\rm (J3)$ is redundant. 

\begin{theorem} \label{T2.2}
    The following axioms form a 4-base for $\mathbb{BA}$. 
\begin{enumerate}
   \item[{\rm(J1)}] $x \land y \approx y \land x,$
     \item[{\rm(J2)}] ${x \land ( y \land z )\approx (x \land y) \land z},$
   \item[{\rm(J4)}]  ${x'' \approx x},$
   \item[{\rm(J5)}] ${x' \approx (x \land y)' \land (x \land y')'}$.
\end{enumerate}
\end{theorem}

        The proof of the above theorem depends on the following lemmas, where (J1), (J2), (J4) and (J5) are assumed as hypotheses. In view of (J2), we frequently use $x \land y \land z$ for $x \land (y \land z)$ and $(x \land y) \land z$.  Moreover,  (J1) and (J2) will be used frequently in the proofs without being mentioned explicitly.

\begin{lemma} \label{23} 
    ${x \land (x \land y')' \approx y \land (y \land x' )'}$.    
\end{lemma}
\begin{proof}
\begin{align*}
    x \land (x \land y')'
    &= [(x' \land y)' \land (x' \land y')'] \land (x \land y')'  &\text{by axiom (J5) and (J4)}\\
    &=(x' \land y)' \land [(y' \land x)' \land (y' \land x')]'\\
    &= (x' \land y)' \land y'' &\text {by axiom (J5)}\\
    &= (x' \land y)' \land y &\text {by axiom (J4)}\\
    &= y \land (y \land x')'.
\end{align*}
Hence, the lemma is proved.
\end{proof}

\begin{lemma} \label{24} 
    $x \land x' = y \land y'$.
\end{lemma}
\begin{proof}
\begin{align*}
               x \land x' 
               &= x \land [(x \land y)' \land (x \land y')']  &\text{by axiom (J5)}\\
              &= (x \land y)' \land [x \land (x \land y')'] &\text{by axiom (J1) and (J2)}\\
             &= (x \land y)' \land  [y \land (y \land x')'] &\text{by Lemma \ref{23}}\\
            &= y \land [(y \land x)' \land (y \land x')'] &\text{by axiom (J1) and (J2)}\\
           &= y \land y' &\text{by axiom (J5)},
\end{align*}
proving the lemma.
\end{proof}

    It should be mentioned that Meredith and Prior ~\cite{cM68} (on page 213) had a different proof of the above lemma.

In view of the above lemma, the following definition is unambiguous.
\begin{definition} \label{25} 
    We set $x \land x' := 0$.
\end{definition}

\begin{lemma} \label{28} 
    $0' \land (x \land x)'=x'$.
\end{lemma}
\begin{proof}
  $x' = (x \land x')' \land (x \land x'')'= 0' \land (x \land x)' \text{ by (J5)}$.
\end{proof}

\begin{lemma} \label{210} 
    $(x \land y)' = 0' \land (x \land y \land x)'$.
\end{lemma}
\begin{proof}
\begin{align*}
    (x \land y)' &= [(x \land y) \land x]' \land [(x \land y) \land x']' &\text{by (J5)}\\
                 &= (x \land y \land x)' \land (y \land 0)'  &\text{by (J1) and (J2)}\\         
                 &= (x \land y \land x)' \land 0' &\text{by (J5)}\\
                 &= 0' \land (x \land y \land x)',
\end{align*}
which proves the lemma.
\end{proof}

\begin{lemma} \label{211} 
    $0' \land [(x \land y) \land (x \land y)]' = [y \land (x \land y)]' $.
\end{lemma}

\begin{proof}
\begin{align*}
    [y \land (x \land y)]'
       &=[(y \land x \land y)  \land x]' \land [(y \land x \land y) \land x']'
        &\text{by (J5)}\\
        &=(y \land x \land y \land x)' \land 0' \\
        &= 0' \land (y \land x \land y \land x)'  &\text{by (J1)},
\end{align*}
proving the lemma.
\end{proof}

\begin{lemma}\label{212} 
    $(x \land y \land x)' = (y \land x)' $.
\end{lemma}

\begin{proof}
\begin{align*}
     (x \land y \land x)' &= 0' \land [y \land x \land y \land x]'  &\text{ by Lemma \ref{211}}\\
                                  &= 0' \land [x \land y \land x \land y]'  \\  
                                  &= 0' \land [x \land (y \land y) \land x]' \\
                                   &= 0' \land [x \land (y \land y)]' &\text{ by Lemma \ref{210}},\\
                                   &= 0' \land [y \land (x \land y)]' \\
                                   &= (y \land x)'  &\text{ by Lemma \ref{210}},
\end{align*}
Hence the lemma.
\end{proof}

\medskip
We are now ready to prove Theorem \ref{T2.2}\\

\begin{proof}{\bf of Theorem \ref{T2.2}}:\\ 
We need to prove the identity $x \land x =x  $.\\
Now,
\begin{align*}
   x \land x &= (x \land x)'' \\
                &= [0' \land (x \land x \land x)']'  &\text{ by Lemma \ref{210}}\\
                  &= [0' \land (x \land x)']' &\text{ by Lemma \ref{212}}\\
                  &= x''&\text{ by Lemma \ref{28}}
                   &= x,\\
\end{align*} 
proving Theorem \ref{T2.2}.
\end{proof}

Thus, the third axiom of Johnson is redundant.

\vspace{1cm}

\section{Identities of Associative Type of length 3}

\medskip
\vspace{.5pt}
\par
Since the axioms in Theorem 2.1 (or Theorem \ref{T2.2}) include the associative law, which is one of the identities of associative type of length 3 (see definition below), the following question presents itself naturally:\\

{\bf Question A:} For every identity I of associative type of length 3, is there a base for the variety of Boolean algebras containing the identity I ? \\  

We will prove that the answer to Question A is in the affirmative for identities of associative type of length 3.  In fact, we will show that corresponding to each of the 14 identities of associative type of length 3 (see below), there is a base for Boolean algebras containing that identity. It turns out that, 
of these bases, some are actually 2-bases and others are 3-bases. 
A detailed account of these will be given in a sequence of papers of which the present paper is the first one and \cite{Sa24} will be the sequel. 

For the benefi to the reader, we will now recall the definition of an identity of associative type of length 3, 
from \cite{CoSa18}, generalizing the associative law. 

\begin{definition} 
An identity $p \approx q$ is of {\it asociative type of length 3} if the terms $p$ and $q$ will both contain three distinct variables, say $x,y,z$, that occur in any order and that are grouped in one of the two (obvious) ways.  
\end{definition}

For a more detailed discussion about these and the ones of length $n$, in general, please refer to \cite{CoSa18, CoSa19}.  

The following identities of associative type of length 3 play a crucial role in the next proposition.

Let $\Gamma$ denote the set consisting of the following $14$ identities of associative type (of length 3) in the  language $\langle \land \rangle$, where $\land $ is a binary operation:
\begin{thlist}   
\item[A1]   $x \land (y \land  z) \approx (x \land   y) \land z$   
\item[A2]  $x \land  (y \land z) \approx x \land (z \land y)$
\item[A3]   $x \land  (y \land  z) \approx (x \land  z) \land y$ 
\item[A4]   $x \land  (y \land  z) \approx y \land  (x \land  z)$
\item[A5]   $x \land  (y \land z) \approx (y \land  x) \land z$
\item[A6]    $x \land (y \land  z) \approx y \land  (z \land  x)$
\item[A7]    $x \land  (y \land z) \approx (y \land  z) \land x$
\item[A8]   $x \land (y \land  z) \approx  (z \land  x) \land y$
\item[A9]    $x \land  (y \land  z) \approx z \land  (y \land  x)$
\item[A10]   $x \land  (y \land  z) \approx (z \land  y) \land x$
\item[A11]    $(x \land y) \land z \approx (x \land  z) \land  y$ 
\item[A12]   $(x \land  y) \land  z \approx (y \land  x) \land z$ 
\item[A13]    $(x \land  y) \land z \approx (y \land  z) \land x$ 
\item[A14]   $(x \land y) \land z \approx (z \land y) \land  x$. 
\end{thlist}

The following proposition is crucial for us and is easily provable, or see \cite{CoSa18} for a proof.

\begin{proposition} \label{propo_neces_identities} 
Let $\mathbb G$ be the variety of all groupoids of type $\langle \land \rangle$, $\land$ being binary, and let $\mathbb{A}i$ be the subvariety of $\mathbb G$ defined by the identity ${\rm(Ai)} \in \Gamma$.
Let $\mathbb V$ denote the subvariety of $\mathbb G$ defined by a single identity of associative type of length 3.  Then $ \mathbb V =\mathbb Ai$, for some $i \in \{1, 2, \cdots, 14\}$.  
\end{proposition}

\medskip
\section{A 2-base for Boolean algebras containing the identity (A6) of associative type}

From now on, we will show that, for every identity (Ai) of associative type, there is a base that for $\mathbb{BA}$ that contains (Ai). 
    In this Section, we consider the case when $i=6$.  In other words, we present a 2-base for Boolean algebras, which contains the identity (A6) of associative type.  Interestingly, an identity obtained by a slight modification of the identity (J5) of Johnson, called (J5'), plays an important role in some of the bases that will be presented.
    
\begin{theorem} \label{T4.1}  
    The following identities form a 2-base for the variety of Boolean algebras:
\begin{enumerate}
   \item[{\rm(A6)}] $x \land (y \land  z) \approx y \land  (z \land  x)$,
   \item[{\rm(J5')}] ${x \approx (x' \land y)' \land (x' \land y')'}$.
\end{enumerate}
\end{theorem}

   We now present a sequence of lemmas leading to the proof of the above theorem.
  But, first, we note that (A6) is clearly equivalent to the following identity:
    
    (A6') ${x \land ( y \land z )\approx z \land (x \land y)}$. 
    
    {\bf In the following lemmas we assume that} (A6') {\bf and} (J5') {\bf are given.}

\begin{lemma} \label{42}
    $x \land [(y \land z) \land u] = u \land [z \land (x \land y)]$.
\end{lemma}

\begin{proof}
\begin{align*}
   x \land [(y \land z) \land u] &= u \land [x \land (y \land z)] &\text{by (A6')}\\
                                 &= u\land [ z \land (x \land y)] &\text{by (A6')},
\end{align*}
which proves the lemma.
\end{proof}

\begin{lemma}\label{43}
    $(x \land y) \land (z \land u) = x \land [(y \land u) \land z]$.
\end{lemma}
\begin{proof}
\begin{align*}
    (x \land y) \land (z \land u) 
             &= u \land [(x \land y) \land z] &\text{by (A6')}\\
             &= z \land [y \land (u \land x)] &\text{by Lemma \ref{42}}\\ 
             &= z \land [x \land (y \land u)] &\text{by (A6')}\\
             &= x \land [(y \land u) \land z] &\text{by (A6')},
\end{align*}
proving the lemma.
\end{proof}

\begin{lemma}\label{44}
    $x \land [(y \land z) \land u] = y \land [(z \land x) \land u]$.
\end{lemma}

\begin{proof}
\begin{align*}
   x \land [(y \land z) \land u] 
                  &= u \land [x \land (y \land z)] &\text{by (A6')}\\
                  &= u \land [y \land (z \land x)] &\text{by (A6')}\\
                  &= y \land [(z \land x) \land u] &\text{by (A6')},
\end{align*}
hence the proof is complete.
\end{proof}

\begin{lemma} \label{45}
    $(x \land y) \land (z \land u) = z \land [(y \land (u \land x)]$.
\end{lemma}

\begin{proof}
\begin{align*}
    z \land [y \land (u \land x)]
            &= z \land [x \land (y \land u)] &\text{by (A6')}\\
            &= z \land [u \land (x \land y)] &\text{by (A6')}\\
            &= u \land [(x \land y) \land z] &\text{by (A6')}\\
            &= (x \land y) \land (z \land u) &\text{by (A6')},
 \end{align*}
proving the lemma.
\end{proof}

\begin{lemma}\label{46}
    $x \land [(y \land z) \land u] = u \land [z \land (x \land y)]$.
\end{lemma}

\begin{proof}
\begin{align*}
    x \land [(y \land z) \land u]
                 &= u \land [x \land (y \land z)] &\text{by (A6')}\\
                 &= u \land [z \land (x \land y)] &\text{by (A6')}.
\end{align*}
Hence, the proof is complete.
\end{proof}

\begin{lemma}\label{47}
    $(x \land y) \land (z \land u) = y \land [(u \land x) \land z]$.
\end{lemma}
\begin{proof}
\begin{align*}
   y \land [(u \land x) \land z]
             &= z \land [y \land (u \land x)] &\text{by (A6')}\\
             &=z \land [u \land (x \land y)] &\text{by (A6')}\\
             &= (x \land y) \land (z \land u) &\text{by (A6')},
\end{align*}
whence the lemma is proved. 
\end{proof}

\begin{lemma}\label{48}
    $(x \land y) \land (z \land u) = x \land [(y \land u) \land z]$.
\end{lemma}

\begin{proof}
\begin{align*}
    x \land [(y \land u) \land z] &= z \land [u \land (x \land y)] &\text{by Lemma \ref{42}} \\ 
                                  &= (x \land y) \land (z \land u) &\text{by (A6')}.
\end{align*}
Hence the proof is complete.
\end{proof}

\begin{lemma}\label{49}
    $(x' \land y)' \land [(x' \land y')' \land z] = z \land x$.
\end{lemma}

\begin{proof}
\begin{align*}
    (x' \land y)' \land [(x' \land y')' \land z]
             &= z \land [(x' \land y)' \land (x' \land y')'] &\text{by (A6')}\\
             &= z \land x &\text{by (J5')},
\end{align*}
proving the lemma.
\end{proof}

\begin{lemma}\label{410}
    $[x \land (y' \land z)'] \land [u \land (y' \land z')'] = x \land (y \land u)$.
\end{lemma}

\begin{proof}
\begin{align*}
    [x \land (y' \land z)'] \land [u \land (y' \land z')']
         &= x \land [ \{(y' \land z)' \land (y' \land z')'\} \land u] &\text{by Lemma \ref{48}}\\         
         &= x \land (y \land u) &\text{by (J5')},
\end{align*}
completing the proof.
\end{proof}

\begin{lemma}\label{411}
    $[(x \land y) \land z] \land u = (u' \land t)' \land [(x \land y) \land \{z \land (u' \land t')'\}]$.
\end{lemma}

\begin{proof}
\begin{align*}
    [(x \land y) \land z] \land u
           &= (u' \land t)' \land [(u' \land t')' \land \{(x \land y) \land z \}]
                          &\text{by Lemma \ref{49}}\\   
          &= (u' \land t)' \land [z \land \{(u' \land t')' \land (x \land y)\}]
           &\text{by (A6')}\\
           &= (u' \land t)' \land [z \land \{x \land (y \land (u' \land t')')\}]
           &\text{by (A6')}\\
           &= (u' \land t)' \land [z \land\{(u' \land t')' \land (x \land y )\}]
           &\text{by (A6')}\\
           &= (u' \land t)' \land [(x \land y) \land \{z \land (u' \land t')'\}]
           &\text{by (A6')},
\end{align*}
which proves the lemma.
\end{proof}

\begin{lemma}\label{412}
  $[(x \land y) \land z] \land u = x \land [u \land (z \land y)]$.  
\end{lemma}

\begin{proof} 
\begin{align*}
    [(x \land y) \land z] \land u \\
        &{\hspace{-2cm}}=(u' \land t)' \land [(x \land y) \land \{z \land (u' \land t')' \}] &\text{by Lemma \ref{411}}\\  
        &{\hspace{-2cm}}= x \land [\{y \land (u' \land t)'\} \land \{z \land (u' \land t')' \}  &\text{by Lemma \ref{44}}\\ 
        &{\hspace{-2cm}}= x \land [y \land (u \land z)] &\text{by Lemma \ref{410}}\\ 
        &{\hspace{-2cm}}= x \land [z \land (y \land u)] &\text{by (A6')}\\
        &{\hspace{-2cm}}= x \land [u \land (z \land y)] &\text{by (A6')}, 
\end{align*}
completing the proof.
\end{proof}

\begin{lemma} \label{412a}
    $(x' \land y)' \land [z \land \{(x' \land y')' \land u \}] = (u \land z) \land x$.
\end{lemma}

\begin{proof}
\begin{align*}
    (x' \land y)' \land [z \land \{(x' \land y')' \land u \}]\\
               &= (u \land z) \land [(x' \land y)' \land (x' \land y')'] \text{ by Lemma \ref{45}}\\ 
               &= (u \land z) \land x  \quad \text{  by (J5')}
\end{align*}
\end{proof}

\begin{lemma}\label{413}
    $(x \land y) \land z = x \land [\{y \land (z' \land u')'\} \land (z' \land u)']$.
\end{lemma}

\begin{proof}
\begin{align*}
    (x \land y) \land z &= (z' \land u)' \land [(z' \land u')' \land (x \land y)] &\text{ by Lemma \ref{49}} \\      
         &= x \land [\{y \land (z' \land u')'\} \land (z' \land u)'] &\text{ by Lemma \ref{42}},  
\end{align*}
yielding the proof.
\end{proof}

\begin{lemma}\label{414}
    $(x' \land y)' \land [\{(x' \land y')' \land z\} \land u] = z \land (x \land u)$.
\end{lemma}

\begin{proof}
\begin{align*}
    (x' \land y)' \land [\{(x' \land y')' \land z\} \land u]
          &= u \land [(x' \land y)' \land \{(x' \land y')'  \land z\} &\text{ by (A6')}\\
          &=(x' \land y')' \land [\{z \land (x' \land y)'\} \land u]  &\text{ by (A6')}\\
          &= u \land [(x' \land y')' \land \{z \land (x' \land y)'\}] &\text{ by (A6')}\\ 
          &= u \land [(x' \land y)' \land \{(x' \land y')' \land z\}] &\text{ by (A6')}\\            
          &= u \land [z \land \{(x' \land y)' \land (x' \land y')'\}]  &\text{ by (A6')}\\                    
          &= u \land (z \land x) &\text{ by (B2)}\\
          &= x \land (u \land z) &\text{ by (A6')}\\
          &= z \land (x \land u) &\text{ by (A6')},
\end{align*}
proving the lemma.
\end{proof}

\begin{lemma}\label{415}
    $[(x \land y) \land z] \land u = x \land [z \land (y \land u)]$.
\end{lemma}

\begin{proof}
\begin{align*}
     [(x \land y) \land z] \land u 
           &= x \land [u \land (z \land y)] &\text{ by Lemma \ref{412}} \\ 
           &= x \land [y \land (u \land z)] &\text{ by (A6')}\\
           &= x \land [z \land (y \land u)] &\text{ by (A6')}\\
           &= x \land [u \land (z \land y)] &\text{ by (A6')},
\end{align*}
whence the lemma is proved.
\end{proof}

\begin{lemma}\label{416}
    $(x \land y) \land z = (z \land y) \land x$.
\end{lemma}

\begin{proof}
\begin{align*}
   (z \land y) \land x
            &= (x' \land y)' \land [y \land \{(x' \land y')' \land z\}] &\text{ by  Lemma \ref{412a}}\\        
        &=[\{(x' \land y)' \land (x' \land y')'\} \land y] \land z  &\text{ by  Lemma \ref{415}}\\ 
        &= (x \land y) \land z &\text{ by (J5')},
\end{align*}
which proves the lemma.
\end{proof}

\begin{lemma} \label{417}
    $[x \land (y' \land z')'] \land (y' \land z)' = y \land x$.
\end{lemma}

\begin{proof}
 \begin{align*}
      [x \land (y' \land z')'] \land (y' \land z)' \\
            &= [(y' \land z)' \land (y' \land z')'] \land x &\text{ by  Lemma \ref{416}}\\ 
            &= y \land x, &\text{ by (J5')}, 
\end{align*}
thus the proof is complete.
\end{proof}

\begin{lemma} \label{418}
    $(x \land y) \land z  = x \land (z \land y)$.
\end{lemma}

\begin{proof}
\begin{align*}
    (x \land y) \land z
         &= x \land [\{y \land (z' \land u')'\} \land (z' \land u)'] &\text{ by  Lemma \ref{414}}  \\           
         &= x \land (z \land y) &\text{ by lemma \ref{417}},  
\end{align*}
whence the lemma is proved.
\end{proof}

\begin{lemma}\label{419}
    $x \land y  = y \land x$.
\end{lemma}
\begin{proof}
\begin{align*}
    y \land x &= [x \land (y' \land z')'] \land (y' \land z)' &\text{ by Lemma \ref{417}}\\    
              &= x \land [(y' \land z)' \land (y' \land z')'] &\text{ by  Lemma \ref{418}} \\   
              &= x \land y &\text{ by (J5')},
\end{align*}
which proves the lemma.
\end{proof}
\begin{corollary} \label{420}
    $x \land (y \land z)  = (x \land y) \land z$.
\end{corollary}

\begin{proof}
     Use the preceding two lemmas.
\end{proof}

Next, we will show that $x''=x$.

\begin{lemma} \label{421}
    $x \land x'= y \land y'$.
\end{lemma}
\begin{proof}
\begin{align*}
    x \land x' 
         &=[(x' \land y')' \land (x' \land y'')'] \land [(x'' \land y')' \land (x'' \land y'')']
              & \text{  by (J5')}\\
              &=[(y' \land x')' \land (y'' \land x')'] \land [(y' \land x'' )' \land (y'' \land x'')']   
              &\text{ by Lemma \ref{419}} \\
              &=[(y' \land x')' \land (y' \land x'' )' ] \land [(y'' \land x')' \land (y'' \land x'')'] \\
              &\hspace{.5cm} \text{ by Lemma \ref{419} and 
              Corollary \ref{420}}   \\
               &=y \land y' &\text{  by (J5')},
\end{align*}
proving the lemma.
\end{proof}

In view of the preceding lemma, we can make the following definition.
\begin{definition} \label{423}
    Let $x \land x' := 0$.
\end{definition}

\begin{lemma} \label{424}
    $x''=x$.
\end{lemma}
\begin{proof}
    Since $(x' \land x'')' \land (x' \land x''')' = x$, we have $0' \land (x' \land x''')' = x$.  Also,
    from $(x''' \land x')' \land (x''' \land x'')' = x''$, we get $(x''' \land x') \land 0' = x''$.  Hence it follows that $x''=x$.
\end{proof}

\begin{proof} {\bf of Theorem \ref{T4.1}:} \\
Observe that, in view of Lemma \ref{419} and Lemma \ref{420}, we see that (J1) and (J2) hold, while (J4) holds in view of Lemma \ref{424}.  It is also easy to see that (J5') and Lemma \ref{424} imply (J5).  
Thus, the axioms of Theorem \ref{T4.1} imply the axioms of Theorem \ref{T2.2}.  For the converse, it suffices to observe that the Boolean algebra $\mathbf{2}$ satisfies the axioms of Theorem \ref{T4.1}.
Hence the proof of Theorem \ref{T4.1} is complete. 
\end{proof}

\medskip

\section{A 2-base for Boolean algebras containing the identity (A8)}

      In this section we present a 2-base containing the identity (A8) of associative type.
    
\begin{theorem} \label{T5.1}
    The following identities form a 2-base for Boolean algebras:
\begin{enumerate}
   \item[{\rm(A8)}] ${x \land ( y \land z )\approx (z \land x) \land y}$  
   \item[{\rm(J5')}] ${x \approx (x' \land y)' \land (x' \land y')'}$.  
\end{enumerate}
\end{theorem}
      We now present a sequence of lemmas leading to the proof of the above theorem.

{\bf In the following lemmas we assume that} (A8) {\bf and} (J5') {\bf are given.}

\begin{lemma} \label{L0}  
 $(x \land y) \land (z \land u) = y \land [u \land (x \land z)].$  
\end{lemma}

\begin{proof}
\begin{align*}
(x \land y )\land (z  \land u)
       &= y \land [(z \land u) \land x] &\text{by (A8)}\\
       &= y \land [u \land (x \land z]] &\text{by (A8)},
\end{align*}
whence the lemma.
\end{proof}

\begin{lemma} \label{L1} 
 $[(x \land y) \land z] \land u = (x \land z) \land (y \land u).$  
\end{lemma}

\begin{proof}
\begin{align*}
    (x \land z) \land (y \land u)] \\
         &= z \land \{(y \land u) \land x\} &\text{ by (A8)}\\
         &= z \land [u \land (x \land y)] &\text{ by (A8)}\\
         &= [(x \land y) \land z] \land u &\text{ by (A8)},
\end{align*}
completing the proof.
\end{proof}

\begin{lemma} \label{L2} 
$x \land y = (x' \land w')' \land  [y \land (x' \land w)'] $  
\end{lemma}

\begin{proof}
\begin{align*}
x \land y &= [(x' \land w)' \land (x' \land w')'] \land y &\text{by (J5')}\\
&=(x' \land w')' \land  [y \land (x' \land w)']  &\text{by (A8)},
\end{align*}
proving the lemma.
\end{proof}

\begin{lemma} \label{L3} 
 $(x \land z) \land (y \land u) = [(x' \land w')' \land z] \land [\{y \land (x' \land w)'\} \land u]. $  
\end{lemma}

\begin{proof}
\begin{align*}
 (x \land z) \land (y \land u) &= [(x \land y) \land z] \land u &\text{by Lemma \ref{L1}}\\ 
                                                   &= [\{(x' \land w')' \land (y \land (x' \land w'))\} \land z] \land u  &\text{by Lemma \ref{L2}}\\ 
                                                    &= [(x' \land w')' \land z] \land [\{y \land (x' \land w)'\} \land u] &\text{by Lemma \ref{L1}},   
\end{align*}
proving the lemma.
\end{proof}

\begin{lemma} \label{L4} 
 $[(x \land y) \land z] \land (u \land w) = (z \land x) \land [w \land (y \land u)].$ 
\end{lemma}

\begin{proof}
\begin{align*}   
[(x \land y) \land z)] \land (u \land w) &=  [y \land (z \land x)] \land (u \land w) &\text{by (A8)}\\
                                                          &= (z \land x) \land [w \land (y \land u)] &\text{by Lemma \ref{L0}} 
\end{align*}
which completes the proof.
\end{proof}

\begin{lemma} \label{L5} 
$[(x \land y) \land z] \land u = (z \land x) \land (u \land y).$ 
\end{lemma}

\begin{proof}
\begin{align*}   
[(x \land y) \land z] \land u &= [y \land (z \land x)] \land u  &\text{by (A8)}\\     
                                          &= (z \land x) \land (u \land y) &\text{by (A8)},    
\end{align*}
hence the lemma.
\end{proof}

\begin{lemma} \label{L6} 
 $(x \land y) \land [z \land (u \land w)] = (y \land x) \land (w \land (z \land u)).$  
\end{lemma}

\begin{proof}
\begin{align*}
(x \land y) \land [z \land (u \land w)] &=  [x \land z) \land y] \land (u \land w) &\text{by Lemma \ref{L1}}\\ \ 
                                                         &= (y \land x) \land [w \land (z \land w)  &\text{by Lemma \ref{L4}}                 
\end{align*}
completing the lemma.
\end{proof}

\begin{lemma} \label{L7} 
 $(x \land y) \land (z \land u) = (y \land x) \land (u \land z).$  
\end{lemma}

\begin{proof}
\begin{align*}
(x \land y) \land (z \land u) &= [(x \land z) \land y] \land u &\text{by Lemma \ref{L1}}\\ 
                                           &= (y \land x) \land (u \land z) &\text{by Lemma \ref{L5}}  
\end{align*}
\end{proof}

\begin{lemma} \label{L8} 
 $(x \land y) \land [z \land (u \land w))]= (x \land y) \land [(z \land u) \land w].$ %
\end{lemma}

\begin{proof}
\begin{align*}
(x \land y) \land [z \land (u \land w)] &= (y \land x) \land [w \land (z \land u)] &\text{{by Lemma \ref{L6} (18)}}\\ 
                                                         &= (x \land y) \land [(z \land u)] \land w] &\text{by Lemma \ref{L7}}  \\
\end{align*}
proving the lemma.
\end{proof}

\begin{lemma} \label{L9} %
$[(x' \land y')' \land z] \land [u \land \{(x' \land y)' \land  w\}] = (x \land z) \land (u \land w).$ %
\end{lemma}

\begin{proof}
\begin{align*}
[(x' \land y')' \land z] \land [u \land \{(x' \land y)' \land  w\}]
 &=  [(x' \land y')' \land z] \land [\{u \land (x' \land y)' \} \land  w] \\
 &\text{\qquad by Lemma \ref{L8}}\\ 
 &= (x \land z) \land (u \land w).\\ 
 &\text{\qquad by Lemma \ref{L3}}, 
\end{align*}
whence the lemma.
\end{proof}

\begin{lemma} \label{L10}  
 $[(x \land y) \land z] \land (u \land w) = z \land [w \land \{y \land (u \land x)\}]. $ %
\end{lemma}

\begin{proof}
\begin{align*}
[(x \land y) \land z] \land (u \land w) &= z \land [(u \land w) \land (x \land y)] &\text{by (5)}   \\                                                          
                                                         &= z \land [w \land \{y \land (u \land x)\}] &\text{by Lemma \ref{L0}}, 
\end{align*}
completing the proof.
\end{proof}

\begin{lemma} \label{L11} 
 $ (x \land y) \land [z \land (u \land w)] = y \land [w \land \{z \land (u \land x)\}].$  %
\end{lemma}

\begin{proof}
\begin{align*}
 (x \land y) \land [z \land (u \land w)]&= [(x \land z) \land y] \land (u \land w)  &\text{by Lemma \ref{L1}}\\ 
                                                          &= y  \land [w \land \{z \land (u \land x))) &\text{by Lemma \ref{L10}}, %
\end{align*}
proving the lemma.
\end{proof}

\begin{lemma} \label{L12} 
 $[(x' \land y')' \land z] \land [u \land \{(x' \land y)' \land  w\}] = z \land [w \land (u \land x)].$  %
\end{lemma}

\begin{proof}
\begin{align*}
[(x' \land y')' \land z] \land [u \land \{(x' \land y)' \land  w\}]  &= z \land [w \land \{u \land ((x' \land y)' \land (x'  \land y')')\}] \\
                                                                         &\text{ \qquad \qquad \qquad \qquad by Lemma \ref{L11}}\\
                                                                 & = z \land [w \land (u \land x)] \text{\qquad \ by (A8)}, 
\end{align*}
completing the proof.\end{proof}

\begin{lemma} \label{L13} %
 $(x \land y) \land (z \land u) = y \land [u \land (z \land x)].$  %
\end{lemma}

\begin{proof}
\begin{align*}
 (x \land y) \land (z \land u) &= [(x' \land t)' \land y] \land [z \land \{(x' \land t)' \land u\}] &\text{by Lemma \ref{L9}}\\ %
                                                 &= y \land (u \land (z \land x)) &\text{by Lemma \ref{L12}}, %
\end{align*}
hence the lemma.
\end{proof}

\begin{lemma} \label{L14}  %
$(x \land (y \land z)) \land u = y \land (u \land (z \land x)). $  
\end{lemma}

\begin{proof}
\begin{align*}
(x \land (y \land z)) \land u &= [(z \land x) \land y] \land u &\text{by (A8)} \\
                                           &= y \land [u \land (z \land x)]  &\text{by (A8)},
\end{align*}
thus the lemma is proved.
\end{proof}

\begin{lemma} \label{L15} 
 $x \land (y \land (z \land u)) = (x \land z) \land (y \land u). $ 
\end{lemma}

\begin{proof}
\begin{align*}
x \land (y \land (z \land u)) &= [u \land (x \land z)] \land y &\text{by Lemma \ref{L14}}\\ 
                                            &= (x \land z) \land (y \land u) &\text{by (A8)},
\end{align*}
proving the lemma.
\end{proof}

\begin{lemma} \label{L16} 
 $x \land [y \land (z \land u)] = u \land [z \land (x \land y)].$ 
\end{lemma}

\begin{proof}
\begin{align*}
x \land [y \land (z \land u)]  &=  (x \land z) \land (y \land u) &\text{by Lemma \ref{L15}}\\ 
                                           &= z \land [u \land (y \land x)] &\text{by Lemma \ref{L13}}\\    
                                           &= (z \land y) \land (u  \land  x) &\text{by Lemma \ref{L15}}\\ 
                                           &= y \land [x \land (u \land z)] &\text{by Lemma \ref{L13}}.
\end{align*}
 On the other hand,
\begin{align*} 
u \land [z \land (x \land y)] &= (y \land u) \land (x \land z) &\text{by Lemma \ref{L13}}\\ 
                                           &= y \land [x \land (u \land z)] &\text{by Lemma \ref{L15}}.
\end{align*}
Hence the lemma follows.
\end{proof}

\begin{lemma} \label{L17} 
 $ x \land y = (x' \land z')' \land (y \land (x' \land z)').$  
\end{lemma}

\begin{proof}
\begin{align*}
x \land y &= [ (x' \land z')' \land (x \land z)'] \land y  &\text{by (J5')} \\
               &= (x' \land z')' \land (y \land (x' \land z)')  &\text{by (A8)},
\end{align*}
whence the lemma.
\end{proof}

\begin{lemma} \label{L18} 
 $x \land (y \land z) = (y' \land u)' \land (z \land  (x \land  (y' \land  u')')).$  
\end{lemma}

.
\begin{proof}
\begin{align*}
x \land (y \land z)  &=  x \land [(y' \land u')') \land \{z \land (x' \land u)'\}]  &\text{by Lemma \ref{L17}} \\                                    &= (y' \land u)' \land (z \land (x \land  (y' \land  u')')).  &\text{by Lemma \ref{L16}}, 
\end{align*}
which completes the proof.
\end{proof}

\begin{lemma} \label{L19} 
 $(x \land y)\land (z \land u) = x \land ((u \land z) \land y).$ 
\end{lemma}

\begin{proof}
\begin{align*}
(x \land y)\land (z \land u) &= x \land [z \land (y \land u)]  &\text{by Lemma \ref{L15}}\\ 
                                          &= x \land [(u \land z) \land y]  &\text{by (A8)},
\end{align*}
completing the proof.
\end{proof}

\begin{lemma} \label{L20} 
 $x \land (y \land z) = (x' \land u)' \land ((z \land y) \land (x' \land u')').$ 
\end{lemma}

\begin{proof}
\begin{align*}
x \land (y \land z) &= [(x' \land u)' \land (x' \land u')'] \land (y \land z)  &\text{by (J5')}\\
                             &= (x' \land u)' \land ((z \land y) \land (x' \land u')')   &\text{by Lemma \ref{L19}}, 
\end{align*}
proving the lemma.
\end{proof}

\begin{lemma} \label{L21}  
 $x \land [y \land (z \land u)] = x \land [(z \land y) \land u].$  
\end{lemma}

\begin{proof}
\begin{align*}
x \land [y \land (z \land u)] &=  (u \land x) \land (z \land y) &\text{by Lemma \ref{L13}}\\ 
                                           &= x \land [(z \land y) \land u] &\text{by (A8)},                        \\
\end{align*}
which completes the proof.
\end{proof}

\begin{lemma} \label{L22} 
 $(x' \land y)' \land (z \land (u \land (x' \land y')')) = x \land (z \land u).$  
\end{lemma}

\begin{proof}
\begin{align*}
(x' \land y)' \land (z \land (u \land (x' \land y')')) &= (x' \land y)' \land [z \land \{u \land (x' \land y')']  &\text{by Lemma \ref{L21}}\\ 
                                                                          &= x \land (z \land u) &\text{by lemma \ref{L20}},  
\end{align*}
that completes the proof.
\end{proof}

\begin{lemma} \label{L23} 
 $x \land (y \land z) = y \land (z \land x).$ 
\end{lemma}

\begin{proof}
\begin{align*}
 x \land (y \land z)  &= (y' \land u)' \land (z \land  (x \land  (y' \land  u')')) &\text{by Lemma   \ref{L18}}\\ 
                                 &= y \land (z \land x) &\text{by Lemma \ref{L22}}, 
\end{align*}
proving the lemma. 
\end{proof}
 
 We are now ready to prove Theorem \ref{T5.1}.
 
\begin{proof} {\bf of Theorem \ref{T5.1}}:\\ 

We wish to prove that the axioms (A8) and (J5') imply (A6). 
\begin{align*}
 x \land (y \land z)  &=  y \land (z \land x) &\text{by Lemma \ref{L23}}\\ 
                              &=  z \land (x \land y) &\text{by Lemma \ref{L23}}, 
\end{align*}
proving the identity (A6).  Thus (A6) and (J5') imply the axioms of        
Theorem \ref{T4.1}.  For the converse, note that $\mathbf{2}$ satisfies (A6) and (J5'),
 implying that (A6) and (J5') form a base.
So, the proof of Theorem \ref{T5.1} is complete.  
\end{proof}

\medskip
\section{A 2-base for Boolean algebras containing the identity (A5)}

\begin{theorem} \label{T6.1}
    The following identities form a 2-base for Boolean algebras:
\begin{enumerate}
   \item[{\rm(A5)}] ${x \land ( y \land z )\approx (y \land x) \land z}$, %
   \item[{\rm(J5')}] ${x \approx (x' \land y)' \land (x' \land y')'}$.
\end{enumerate}
\end{theorem}
 
\begin{lemma} \label{62} %
    $x \land [y \land (z \land u)] = x \land [z \land (y \land u)]$.
\end{lemma}
\begin{proof}
\begin{align*}
    x \land [y \land (z \land u)] &= (y \land x) \land (z \land u)  &\text{by (A5)}\\
                                  &= [z \land (y \land x)] \land u  &\text{ by (A5)}\\
                                  &= [(y \land z) \land x] \land u &\text{ by (A5)}\\
                                  &= x \land [(y \land z) \land u] &\text{ by (A5)}\\
                                  &=x \land [z \land (y \land u)]. &\text{ by (A5)},
\end{align*}
proving the lemma.
\end{proof}

\begin{lemma} \label{63} %
    $ x' \land y = [(x' \land y)' \land (x' \land y')]' \land x'  $.
\end{lemma}
\begin{proof}
\begin{align*}
    x' \land y 
        &= [(x' \land y)' \land (x' \land y')]' \land [(x' \land y)' \land (x' \land y')']' &\text{ by (J5')}\\
        &= [(x' \land y)' \land (x' \land y')]' \land x'  &\text{ by (A5)},
\end{align*}
which proves the lemma.
\end{proof}

\begin{lemma} \label{64} %
    $x \land z  = (x' \land y')' \land [(x' \land y)' \land z] $.
\end{lemma}
\begin{proof}
\begin{align*}
    x \land z &= [(x' \land y)' \land (x' \land y')'] \land z &\text{ by (J5')}\\
              &= (x' \land y')' \land [(x' \land y)' \land z]   &\text{ by (A5)},
\end{align*}
thus the proof is complete.
\end{proof}

\begin{lemma} \label{65} %
    $x \land (u \land y) =  x \land [(y' \land z)' \land \{u \land (y' \land z')' \}]$.
\end{lemma}

\begin{proof}
\begin{align*}
    x \land (u \land y) &= x \land [u \land \{(y' \land z)' \land (y' \land z')'\}]  &\text{ by (J5')}\\
              &= x \land [(y' \land z)' \land \{u \land (y' \land z')' \}]  &\text{by Lemma \ref{62}},
\end{align*}
proving the lemma.
\end{proof}

\begin{lemma} \label{66} 
    $(x' \land y) \land z = x' \land [\{(x' \land y)' \land (x' \land y')\}' \land z] $.
\end{lemma}

\begin{proof}
\begin{align*}
    (x' \land y) \land z \\
       &= [\{(x' \land y)' \land (x' \land y')\}' \land \{(x' \land y)' \land (x' \land y')'\}'] \land z  &\text{ by (J5')}\\
       &= [\{(x' \land y)' \land (x' \land y')\}' \land x'] \land z  &\text{ by (J5')}\\
       &= x' \land [\{(x' \land y)' \land (x' \land y')\}' \land z]  &\text{ by (A5)},
\end{align*}
thus proving the lemma.
\end{proof}

\begin{lemma} \label{67} 
    $x \land (z \land u) = (x' \land y')' \land [z \land \{(x' \land y)' \land u\}]  $.
\end{lemma}
\begin{proof}
\begin{align*}
    x \land (z \land u) \\
       &= [(x' \land y)' \land (x' \land y')']  \land (z \land u)  &\text{ by (J5')}\\
       &= (x' \land y')' \land [(x' \land y)' \land (z \land u)]  &\text{ by (A5)}\\
       &= (x' \land y')' \land [z \land \{(x' \land y)' \land u\}] &\text{by Lemma \ref{62}},
\end{align*}
which proves the lemma.
\end{proof}

\begin{lemma} \label{68} 
    $x \land (y \land u) = x \land [(y' \land z)' \land \{y' \land z')' \land u\}]  $.
\end{lemma}
\begin{proof}
\begin{align*}
    x \land (y \land u) 
       &= x \land [(y' \land z')' \land \{(y' \land z)' \land u\}] &\text{by Lemma \ref{64}}\\
       &= x \land [(y' \land z)' \land \{(y' \land z')' \land u\}] &\text{by Lemma \ref{62}},
\end{align*}
proving the lemma.
\end{proof}

\begin{lemma} \label{69} 
$y \land (x' \land z)= x' \land (((x' \land y)' \land (x' \land y'))' \land z)$\\
\end{lemma}

\begin{proof}
\begin{align*}
y \land (x' \land z) &= (x' \land y) \land z  &\text{ by (A5)}\\  
&= x' \land (((x' \land y)' \land (x' \land y'))' \land z) &\text{ by Lemma \ref{66}}.  
\end{align*}
Hence the proof.
\end{proof}

\begin{lemma} \label{189} 
    $x \land (y \land z) = x \land ((((z' \land u)' \land y)' \land ((z' \land u)' \land y'))' \land z).$
\end{lemma}

\begin{proof}     
\begin{align*}  
x \land (y \land z) \\                                              
 &\hspace{-1cm} = x \land [y \land ((z' \land u)' \land (z' \land u')')] \quad \text{ by (J5')}\\
&\hspace{-1cm}  = x \land [(z' \land u)' \land  \{(((z' \land u)' \land y)' \land ((z' \land u)'  \land y'))' \land (z' \land u')'\}] \\ 
&\hspace{-.6cm} \text{ by lemma \ref{69}}\\                                             
&\hspace{-1cm}  =x \land ((((z' \land u)' \land y)' \land ((z' \land u)' \land y'))' \land z) \quad \text{ by lemma \ref{65}},
 \end{align*} 
 proving the lemma.
\end{proof}

\begin{lemma} \label{810} 
    $x \land (y \land u) = x \land [\{((y' \land z')' \land u)' \land ((y' \land z')' \land u')\} \land y]  $.
\end{lemma}
\begin{proof}
\begin{align*}
             x \land (y \land u)    
              &= x \land [(y' \land z)' \land ((y' \land z')' \land u)]\\  &\text{by Lemma 6.8} \\    
          &= x \land [(y' \land z)' \land \{(((y' \land z')' \land u)' \land ((y' \land z')' \land u'))' \land (y' \land z')'\}]\\  &\text{by Lemma 6.4}\\ 
          &= x \land [\{((y' \land z')' \land u)' \land ((y' \land z')' \land u')\}' \land y]\\  &\text{by Lemma 6.5},   
\end{align*}
which proves the lemma.
\end{proof}

\begin{lemma} \label{811} 
    $x \land (y \land u) = x \land (u \land y)$.
\end{lemma}

\begin{proof}
From Lemma  \ref{810} we have  

$x \land (y \land u) = x \land [\{((y' \land z')' \land u)' \land ((y' \land z')' \land u')\} \land y]  $, \\
and from Lemma  \ref{189} we get 

$x \land (u \land y) = x \land [\{((y' \land z')' \land u)' \land ((y' \land z')' \land u')\} \land y]  $. \\
Hence,
 $x \land (y \land u) = x \land (u \land y)$. 
 \end{proof}

\begin{lemma} \label{812}
    $x \land [y \land (z \land u)] =  y \land [x \land (z \land u)]$.
\end{lemma}
\begin{proof}
\begin{align*}
    x \land [y \land (z \land u)]\\
            &= (y \land x) \land (z \land u) &\text{by (A5)}\\
            &= [z \land (y \land x)] \land u) &\text{by (A5)}\\
            &= [z \land (x \land y)] \land u) &\text{by Lemma \ref{811}}\\          
            &= [(x \land z) \land y)] \land u) &\text{by (A5)}\\
            &= y \land [(x \land z) \land u)]  &\text{by (A5)}\\
            &= y \land [x \land (z \land u)]  &\text{by (A5}
\end{align*}
\end{proof}

\begin{proof} {\bf of Theorem \ref{T6.1}} :\\
We prove that (A5) and (J5') imply (A6):
    $x \land (y \land z) = z \land (x \land y)$.
\begin{align*}
    x \land (y \land z)
          &= x \land (z \land y) &\text{by Lemma \ref{811}} \\   
          &= x \land [z \land \{(y' \land t)' \land (y' \land t')']  &\text{by (J5')} \\ 
          &= z \land [x \land \{(y' \land t)' \land (y' \land t')'] &\text{by Lemma \ref{812}}\\    
          &= z \land (x \land y)  &\text{by (A5)}.  
\end{align*}
 Hence, (A6) and (J5') of Theorem \ref{T4.1} hold.  For the converse, note that $\mathbf{2}$ satisfies (A5) and (J5'),
 implying that (A5) and (J5') form a base.  This completes the proof of Theorem \ref{T6.1}.
\end{proof}

\medskip

\section{A 3-base for Boolean Algebras containing the identity (A9)}

       Recall that (A9) is one of the 14 identities of associative type listed in Proposition 3.1.
 We shall now present a 3-base for the variety of Boolean algebras, in which (A9)  occurs as one of the axioms. 
         
\begin{theorem} \label{T8.1}
    The following identities form a 3-base for Boolean algebras:
\begin{enumerate}
   \item[{\rm(A9)}] ${x \land ( y \land z )\approx z \land (y \land x)}$  (J2')   
   \item[{\rm(J4)}]  ${x'' \approx x}$,
   \item[{\rm(J5)}] ${x' \approx (x \land y)' \land (x \land y')'}$.
\end{enumerate}
\end{theorem}
    We first prove the commutative law.   To this end we need the following lemmas. 
    
    {\bf In the following lemmas we assume that  
    axioms (A9), (J4) and (J5) are given as hypotheses.}
    
\begin{lemma} \label{82}
    $ z \land x' = (x \land y')' \land [(x \land y)' \land z]$.
\end{lemma}
\begin{proof}
\begin{align*}
    z \land x'&=  z \land [(x \land y)' \land (x \land y')'] &\text{by (J5)}\\
    &=(x \land y')' \land [(x \land y)' \land z] &\text{by (A9)},
\end{align*}
proving the lemma.
\end{proof}
\begin{lemma} \label{83}
    $ z'  = [x \land (y \land z)]' \land [z \land (y \land x)']' $.
\end{lemma}
\begin{proof}
\begin{align*}
    z' &= [z \land (y \land x)]' \land [z \land (y \land x)']' &\text{by axiom (J5)}\\
       &=(x \land (y \land z)]' \land [z \land (y \land x)']' &\text{by axiom (A9)},
\end{align*}
which proves the lemma.
\end{proof}

\begin{lemma} \label{84}
    $(x \land y)' \land (z \land x') = [(x \land y)' \land z] \land x' $.
\end{lemma}
\begin{proof}
\begin{align*}
    (x \land y)' \land (z \land x')
               &= (x \land y)' \land [(x \land y')' \land \{(x \land y)' \land z\}] &\text{by Lemma \ref{82}}\\
               &= \{(x \land y)' \land z\} \land [(x \land y')' \land (x \land y'')'] &\text{by (J1)}\\
               &= \{(x \land y)' \land z\} \land x' &\text{by (J5)},
\end{align*}
proving the lemma.
\end{proof}

\begin{lemma} \label{85}
    $ z' \land x' = [(x \land (y \land z)]' \land [\{z \land (y \land x)'\}' \land x'] $.
\end{lemma}

\begin{proof}
\begin{align*}
    z' \land x'\\
        &{\hspace{-1cm}}= [\{z \land (y \land x)'\}' \land (z \land (y \land x))'] \land [(x \land (y \land z)')' \land (x \land (y \land z))'] \\
        &{\hspace{-1cm}}= (x \land (y \land z))' \land [(x \land (y \land z)')' \land \{(z \land (y \land x))' \land (z \land (y \land x)')'\}]\\ &\text{   by axiom (A9)}\\
        &{\hspace{-1cm}}= (x \land (y \land z))' \land [(z \land (y \land x)')' \land \{(z \land (y \land x))' \land (x \land (y \land z)')'\}]\\  &\text{   by axiom (A9)} \\
        &{\hspace{-1cm}}= (x \land (y \land z))' \land [(z \land (y \land x)')' \land \{(x \land (y \land z))' \land (x \land (y \land z)')'\}]\\  &\text{   by axiom (A9)} \\
        &{\hspace{-1cm}}= (x \land (y \land z))' \land [(z \land (y \land x)')' \land x'] \quad \text{by axiom (J5)},
\end{align*}
proving the lemma.
\end{proof}
\begin{lemma} \label{86}
    $x' \land y' = y' \land x'$.
\end{lemma}
\begin{proof}
\begin{align*}
    x' \land y'
          &= (y \land (y \land x)\}' \land [\{(x \land (y \land y)'\}' \land y'] &\text{by Lemma \ref{85} }\\
          &= y' \land [\{(x \land (y \land y)'\}' \land \{y \land (y \land x)\}'] &\text{by axiom (A9)}\\
          &= y' \land [\{(x \land (y \land y)'\}' \land \{x \land (y \land y)\}'] &\text{by axiom (A9)}\\
          &= y' \land x'  &\text{by axiom (J5)},
\end{align*}
which proves the lemma.
\end{proof}
\begin{corollary} \label{87}
    $x \land y = y \land x$.
\end{corollary}

We are now ready to prove Theorem \ref{T8.1}.\\  

\begin{proof}{\bf of Theorem \ref{T8.1}}:  \\      
    Now associativity (J2) (or (A1)) follows immediately from (A9) and Corollary \ref{87}.  Hence the axioms (J1), (J2), (J4) and (J5) of Theorem \ref{T2.2} hold. For the converse, note that $\mathbf{2}$ satisfies (A9), (J4) and (J5),
 implying that (A5) and (J5') form a base.  This completes the proof of Theorem \ref{T8.1}.
 \end{proof}

\begin{remark} 
  The identities (A9) and (J5) of Theorem \ref{T8.1}, alone fail to form a base for Boolean algebras, as these two identities hold in the following non-Boolean algebra: \\

\begin{tabular}{r|rr}
$\land$ & 0 & 1\\
\hline
    0 & 0 & 0 \\
    1 & 0 & 0
\end{tabular} \hspace{.5cm}
 \begin{tabular}{r|rr}
$'$:& 0 & 1\\
\hline
   & 0 & 0
\end{tabular} 
\end{remark}

The above remark raises the following problem:\\

{\bf PROBLEM} Investigate the variety of algebras defined by the two identities 
(A9) and (J5).\\

\medskip

\section{A 3-base for Boolean Algebras containing the identity (A1) of associative type}

    We will now present a 3-base for Boolean algebras by modifying the Johnson's 5th axiom, thereby obtaining a further refinement in Johnson's system of five axioms.  Meredith and Prior
remark in their Footnote 5 of \cite{cM68} that Kalman informed them that the following theorem follows from Huntington's fourth set of axioms as given in Huntington ~\cite{eH33i}.   
  
 We give a new proof here.
\begin{theorem} \label{T9.1}
    The following identities form a 3-base for the variety of Boolean algebras:
\begin{enumerate}
   \item[{\rm(J1)}] $x \land y \approx y \land x,$
   \item[{\rm(A1)}] ${x \land ( y \land z )\approx (x \land y) \land z}$ \quad (J2),
   \item[{\rm(J5')}] ${x \approx (x' \land y)' \land (x' \land y')'}$.
\end{enumerate}
\end{theorem}
    The proof is obtained from the following lemmas.  
    
    {\bf The identities (J1), (J2) and (J5') are assumed to be given as hypotheses in the following lemmas.}

\begin{lemma} \label{92}
    $x \land x'= y \land y'$.
\end{lemma}
\begin{proof}
\begin{align*}
    x \land x' 
         &=[(x' \land y')' \land (x' \land y'')'] \land [(x'' \land y')' \land (x'' \land y'')']
                        &\text{by (J5')}\\
               &=[(y' \land x')' \land (y' \land x'')'] \land [(y'' \land x')' \land (y'' \land x'')']\\
               &=y \land y',
\end{align*}
proving the lemma.
\end{proof}

In view of the preceding lemma, the following definition of the element $0$ is unambiguous.
\begin{definition} \label{93}
    Let $0 := x \land x'$.
\end{definition}

\begin{lemma} \label{94}
    $x''=x$.
\end{lemma}
\begin{proof}
    Since $(x' \land x'')' \land (x' \land x''')' = x$, we have $0' \land (x' \land x''')' = x$.  Also,
    from $(x''' \land x')' \land (x''' \land x'')' = x''$, we get $(x''' \land x') \land 0' = x''$.  Since $0' \land (x' \land x''')' = (x''' \land x') \land 0' $ by  (J1), it follows that $x''=x$.
\end{proof}

\begin{proof}{\bf of Theorem \ref{T9.1}}\\
    It is easy to see that (J5') and Lemma \ref{94} imply (J5).  It is also clear that (J1) and (J2) imply (J2'). Thus the axioms 
 (J2)', (J4) and (J5) of Theorem \ref{T8.1} hold.  It is also easy to see that the axioms of Theorem \ref{T8.1} imply the axioms (J1') (J2) and (J5').  Hence the proof of Theorem \ref{T9.1} is complete. 
\end{proof}
 
The question now arises as to whether  the axioms in Theorem \ref{T9.1} are independent?  The following remark answers this question positively.

\begin{remark} It is clear that \{(J1), (J2)\} do not axiomatize Boolean algebras.  
The following identities fail to be  a 2-base for Boolean algebras:\\
(J2) $x \land (y \land z)=(x \land y) \land z$. (A1)  \\
(J5') $x=(x' \land y)' \land (x' \land y')'$\\ 
since they hold in the following algebra: \\

\begin{tabular}{r|rr}
$\land$ & 0 & 1\\
\hline
    0 & 0 & 0 \\
    1 & 1 & 1
\end{tabular} \hspace{.5cm}
\begin{tabular}{r|rr}
$'$ & 0 & 1\\
\hline
   & 0 & 1
\end{tabular} 

\medskip
Also, it is easy to see that the following identities fail to be  a base for Boolean algebras:\\
(J1) $x \land y = y \land x$,   \\
(J5') $x=(x' \land y)' \land (x' \land y')'$\\ 
It is also clear that the identities (J1) and (J2) also fail to be  a base for Boolean algebras.
\end{remark}

The following problems may be of interest.\\
PROBLEM: Investigate the variety defined by (J2) and (J5').\\
PROBLEM: Investigate the variety defined by (J1) and (J5').\\

\medskip
\section{A 2-base for Boolean Algebras containing the identity (A13) of associative type}

\begin{theorem} \label{T10.1} The following identities form a 2-base for Boolean algebras:
\begin{thlist}
\item[A13] $(x \land y) \land z= (y \land z) \land x,$   

\item[J5']    $x=(x' \land y)' \land (x' \land y')'.$
\end{thlist}
\end{theorem}
The proof of Theorem \ref{T10.1} will be obtained through the following lemmas. 

{\bf The identities (A13) and (J5') are assumed to be given as hypotheses in the following lemmas.}

We wish to prove the identity (A5): $x \land (y \land z) = (y \land x) \land z$.

\begin{lemma} \label{1017} 
 $(x \land y) \land (z \land u) = (y \land (x \land z)) \land u.$  
\end{lemma}

\begin{proof}
\begin{align*}                                                            
(x \land y) \land (z \land u)  &= [y \land (z \land u)] \land x &\text{  by (J5')}\\
                                            &= [(z \land u) \land x] \land y  &\text{  by (J5')}\\
                                            &=[(u \land x) \land z] \land y  &\text{  by (J5')}\\
                                            &=[(x \land z)\land u] \land y  &\text{  by (J5')}\\
                                            & =(y \land (x \land z)] \land u &\text{  by (J5')},
\end{align*} 
proving the lemma.                                              
\end{proof}

\begin{lemma} \label{108}
 $((x \land y) \land z) \land u = (u \land (y \land z)) \land x.$  
 \end{lemma}
 
\begin{proof} 
\begin{align*}                                                                 
 ((x \land y) \land z) \land u &= [(y \land z) \land x] \land u &\text{ by (J5')}\\
                                             &= [u \land (y \land z)] \land x &\text{ by (J5')},
 \end{align*} 
 whence the lemma is proved.
 \end{proof}
 
 \begin{lemma} \label{1014} 
   $x \land y = [y \land (x' \land z)'] \land (x' \land z')'.$  
\end{lemma}

\begin{proof}
\begin{align*}
x \land  y &= [(x' \land z)' \land (x' \land z')'] \land y  &\text{ by (J5')}\\
                &= [(y \land (x' \land z)'] \land (x' \land z')'  &\text{ by (A13)},
\end{align*}
completing the proof.
\end{proof}

\begin{lemma} \label{1021}
 $(x \land y) \land z = (x \land (z \land (y' \land u)')) \land (y' \land u')'.$ 
 \end{lemma}

\begin{proof} 
 \begin{align*}                                                                                                
(x \land y) \land z &= [x \land (y' \land u)' \land (y' \land u')'] \land z\\
                             &= [(z \land (y' \land u)'] \land (y' \land u')' \land x 
                                            &\text{ by Lemma \ref{108}}\\ 
                              &=  (x \land (z \land (y' \land u)')) \land (y' \land u')' &\text{ by (J5')},
 \end{align*} 
proving the lemma.
\end{proof}

\begin{lemma} \label{1028}
$((x \land y) \land z) \land u = z \land (y \land (x \land u)). $ 
\end{lemma}

\begin{proof}
\begin{align*}                                               
((x \land y) \land (z \land (u' \land w)')) \land (u' \land w')'\\
          & \hspace{-3cm}= ((y \land (x \land z)) \land (u' \land w)') \land (u' \land w')'  &\text{ by Lemma \ref{1017}} \\
           & \hspace{-3cm}= u \land (y \land (x \land z)) &\text{by Lemma \ref{1014}}
\end{align*}
Hence,
\begin{equation} \label{1027}
 ((x \land y) \land (z \land (u' \land w)')) \land (u' \land w')' = u \land (y \land (x \land z))  
 \end{equation}
We, therefore, have
\begin{align*}
z \land (y \land (x \land u)) 
            &= [(x \land y) \land (u \land (z'  \land w)')] \land (z' \land w')'  &\text{ by (\ref{1027})}\\
               &= ((x \land y) \land z) \land u &\text{ by Lemma \ref{1021}},
\end{align*}
completing the proof.
\end{proof}

\begin{lemma}  \label{1019} 
 $(((x' \land y)' \land z) \land u) \land (x' \land y')' = x \land (z \land u).$  
\end{lemma}

\begin{proof}  
\begin{align*}  
(x \land y) \land (z \land u) &= [(z \land u) \land x] \land y  &\text{ by (A13)}\\
                                               & =[(x \land z) \land u] \land  y  &\text{ by (A13)}
\end{align*}
Thus,
\begin{equation}\label{9}
  (x \land y) \land (z \land u) = ((x \land z) \land u) \land y.  
\end{equation}
Hence,

$(((x' \land y)' \land z) \land u) \land (x' \land y')' 
                                  = [(x' \land y)' \land (x' \land y')'] \land (z \land u)$ \quad by \ref{9}.
\end{proof}

\begin{lemma} \label{1031} 
 $x \land (y \land z) = z \land (y \land x).$  
\end{lemma}

\begin{proof}  
\begin{align*}                      
x \land (z \land u) &= (((x' \land y)' \land z) \land u) ^ (x' \land y')'  &\text{ by Lemma \ref{1019}} \\ 
                              &= u \land (z  \land ((x' \land y)' \land (x' \land y')')) &\text{ by  Lemma \ref{1028}}.
 \end{align*}
 Thus, we have
 \begin{equation} \label{eq:30}
 x \land (z \land u). = u \land (z \land ((x' \land y)' \land (x' \land y')')).  
\end{equation}
From \ref{eq:30}, we have
\begin{align*}              
x \land (y \land z) &= z \land (y \land ((x' \land w)' \land (x' \land w')')) \\
                               &= z \land (y \land x) &\text{  by (J5')}.
\end{align*}
Hence,\\
 $x \land (y \land z) = z \land (y \land x),$ 
 proving the lemma.
\end{proof}

\begin{lemma} \label{1023}  
 $x \land (y \land z) = ((x' \land u')' \land ((x' \land u)' \land y)) \land z. $  
\end{lemma}

\begin{proof}  
\begin{align*}                                                               
[(x \land y) \land z] \land u &= [(z \land x) \land y] \land u\\
                                             &= [u \land (z \land x)] \land y. 
\end{align*}                                                                                   
Hence,
\begin{equation}\label{11}
 ((x \land y) \land z) \land u = [u \land (z \land x)] \land y.  
\end{equation}
Therefore,
\begin{align*} 
x \land (y \land z) &= [(x' \land u)' \and (x' \land u')'] \land (y \land z)  &\text{ by (J5')}\\
                             &= [(y \land z) \land (x' \land u)' ] \land (x' \land u')'  &\text{ by (A13)}\\                                                                                              
                            &=[(x' \land u')' \land \{(x' \land u)' \land y\}] \land z  &\text{ by (\ref{11})},
\end{align*}                                                                            
proving the lemma.
\end{proof}

\begin{lemma} \label{1033} 
 $((x' \land y)' \land z) \land ((x' \land y')' \land u) = x \land (z \land u).$  
\end{lemma}

\begin{proof} 
\begin{align*}   
(y \land z) \land (x \land u).&=  (z \land (y \land x)) \land u  &\text{ by Lemma \ref{1017}} \\                                    
                                            &= (x \land (y \land z)) \land u  &\text{ by \ref{1031}}.
\end{align*}                                            
Hence
\begin{equation}  \label{32}                                                                                     
(x \land (y \land z)) \land u = (y \land z) \land (x \land u).  
\end{equation}
  
\begin{align*}                                                                                                 
x \land (z \land u) &=  ((x' \land y')' \land ((x' \land y)' \land z)) \land u  &\text{ by Lemma \ref{23}}\\
                               &= ((x' \land y)' \land z)) \land  [(x' \land y')' \land u  &\text{ by (\ref{32})},  
\end{align*}
completing the proof.                                                                              
\end{proof}

We are now ready to prove Theorem \ref{T10.1}.

\begin{proof}{\bf of Theorem \ref{T10.1}}:
\begin{align*}                                                                 
((y' \land u)' \land x) \land ((y' \land u')' \land z)
&= [x \land ((y' \land u)'  \land (y' \land u')'] \land z  &\text{ by Lemma \ref{1017}}\\
&=(x \land y) \land z  &\text{ by (J3')}.
\end{align*}
Hence,

$(x \land y) \land z = ((y' \land u)' \land x) \land ((y' \land u')' \land z). $\\ 
Also, by Lemma \ref{1033}, we have

 $y \land (x \land z)=((y' \land u') \land x) \land ((y' \land u')' \land z)$.\\
So, 
  $ (x \land y) \land z  = y \land (x \land z).$  
  Thus,  (A13) and (J5') imply the axioms of Theorem \ref{T6.1}.  The converse being trivial, we conclude that the proof is complete.
  \end{proof}

\medskip
\section{Concluding Remarks}

We would like to conclude this paper by mentioning some of the results whose proofs will appear in the sequel to this paper. 
 
\begin{theorem} The following identities form  a 2-base for Boolean algebras:
\begin{thlist}
\item[A4] $x \land (y \land z)=y \land (x \land z)$   
\item[D] $x=(x' \land y')' \land (x' \land y)'$.
\end{thlist}
\end{theorem}

\begin{theorem} The following identities form  a 2-base for Boolean algebras:
\begin{thlist}
\item[A5] $x \land (y \land z)=(y \land x)\land z $ 
\item[D] $x=(x' \land y')' \land (x' \land y)'$.
\end{thlist}
\end{theorem}

\begin{theorem}The following identities form  a 2-base for Boolean algebras:
\begin{thlist}
\item[A6] $x \land (y \land z)=y \land (z \land x) $ 
\item[D] $x=(x' \land y')' \land (x' \land y)'$.
\end{thlist}
\end{theorem}

\begin{theorem}The following identities form  a 2-base for Boolean algebras:
\begin{thlist}
\item[A8] $x \land (y \land z)=(z \land x) \land y$
\item[D] $x=(x' \land y')' \land (x' \land y)'$.  
\end{thlist}
\end{theorem}

\begin{theorem} The following identities form  a 2-base for Boolean algebras:
\begin{thlist}
\item[A13] $(x \land y) \land z=(y \land z) \land x$
\item[D] $x=(x' \land y')' \land (x' \land y)'$.  
\end{thlist}
\end{theorem}

\begin{theorem} The following is a 3-base for the variety of Boolean algebras:
\begin{thlist}
\item[J5']  $ x = (x' \land y)' \land (x' \land y')'$
\item[A1]  $ x \land (y \land z) = (x \land y) \land z$  
\item[A7]  $x \land (y \land z)=(y \land z) \land x.$            
\end{thlist}
\end{theorem}

 \begin{theorem}  The following is a 3-base for the variety of Boolean algebras:
 \begin{thlist}
\item[J5']  $ x = (x' \land y)' \land (x' \land y')'$
\item[A1] $ x \land (y \land z) = (x \land y) \land z$  
\item[A9]  $x\land (y \land z)=(z \land y) \land x$.   
\end{thlist}
\end{theorem}

 \begin{theorem} The following is a 3-base for the variety of Boolean algebras:
 \begin{thlist}
\item[J5']  $ x = (x' \land y)' \land (x' \land y')'$
\item[A1]  $ x \land (y \land z) = (x  \land y) \land z$   
\item[A10]  $x \land (y \land z)=(z \land y) \land x$.  
\end{thlist}
\end{theorem}

\begin{theorem}  The following is a 3-base for the variety of Boolean algebras:
\begin{thlist}
\item[J5']   $ x = (x' \land y)' \land (x' \land y')'$
\item[A1]  $ x \land (y \land z) = (x \land y) \land z$  
\item[A12]  $(x \land y) \land z=(y \land x) \land z.$  
\end{thlist}
\end{theorem}

\begin{theorem}  The following is a 3-base for the variety of Boolean algebras:
\begin{thlist}
\item[J5'] $x=(x' \land y)' \land (x' \land y')'$
\item[A3] $x \land (y \land z)=(x \land z) \land y $  
\item[A7] $x \land (y \land z)=(y \land z) \land x$.  
 \end{thlist}
 \end{theorem}
 
\begin{theorem} The following is a 3-base for the variety of Boolean algebras:
\begin{thlist}
\item[J5'] $x=(x' \land y)' \land (x' \land y')'$
\item[A3] $x \land (y \land z)=(x \land z) \land y $  
\item[A9] $x \land (y \land z)=(z \land y) \land x$.   
 \end{thlist}
\end{theorem}

\begin{theorem}  The following is a 3-base for the variety of Boolean algebras:
\begin{thlist}
\item[J5'] $x=(x' \land y)' \land (x' \land y')'$
\item[A3] $x \land (y \land z)=(x \land z) \land y $  
\item[A10] $x \land (y \land z)=(z \land y) \land x$.  
 \end{thlist}
\end{theorem}

\begin{theorem}  The following is a 3-base for the variety of Boolean algebras:
\begin{thlist}
\item[J5'] $x=(x' \land y)' \land (x' \land y')'$
\item[A3] $x \land (y \land z)=(x \land z) \land y$  
\item[A12] $(x \land y) \land z=(y \land x) \land z$.  
 \end{thlist}
\end{theorem}

\begin{theorem}  The following is a 3-base for the variety of Boolean algebras:
\begin{thlist}
\item[J5'] $x=(x' \land y)' \land (x' \land y')'$
\item[A3] $x \land (y \land z)=(x \land z) \land y$  
\item[A14] $(x \land y) \land z=(z \land y) \land x. $ 
 \end{thlist}
\end{theorem}

\begin{theorem}  The following is a 3-base for the variety of Boolean algebras:
\begin{thlist}
\item[J5'] $x=(x' \land y)' \land (x' \land y')'$
\item[A7] $x \land (y \land z)=(y \land z) \land x$   
\item[A11] $(x \land y) \land z=(x \land z) \land y.$  
 \end{thlist}
\end{theorem}

\begin{theorem}  The following is a 3-base for the variety of Boolean algebras:
\begin{thlist}
\item[J5'] $x=(x' \land y)' \land (x' \land y')'$
\item[A7] $x \land (y \land z)=(y \land z) \land x$  
\item[A14] $(x \land y) \land z=(z \land y) \land x.$  
 \end{thlist}
\end{theorem}

\begin{theorem}  The following is a 3-base for the variety of Boolean algebras:
\begin{thlist}
\item[J5'] $x=(x' \land y)' \land (x' \land y')'$
\item[A9] $x \land (y \land z)=(z \land y) \land x$  
\item[A11] $(x \land y) \land z=(x \land z) \land y. $ 
 \end{thlist}
\end{theorem}

\begin{theorem}  The following is a 3-base for the variety of Boolean algebras:
\begin{thlist}
\item[J5'] $x=(x' \land y)' \land (x' \land y')'$
\item[A10] $x \land (y \land z)=(z \land y) \land x$  
\item[A11] $(x \land y) \land z=(x \land z) \land y. $  
 \end{thlist}
\end{theorem}

\begin{theorem}  The following is a 3-base for the variety of Boolean algebras:
\begin{thlist}
\item[J5'] $x=(x' \land y)' \land (x' \land y')'$
\item[A11] $(x \land y) \land z=(x \land z) \land y $  
\item[A12] $(x \land y) \land z=(y \land x) \land z.$  
 \end{thlist}
\end{theorem}

\begin{theorem} The following is a 3-base for the variety of Boolean algebras.
\begin{thlist}
\item[J5']   $ x = (x' \land y)' \land (x' \land y')' $
\item[A12] $(x \land y) \land z=(y \land x) \land z$  
\item[c] $x \land (x \land y)' = x \land y'. $    
\end{thlist}
\end{theorem}

\begin{theorem}  The following is a 4-base for the variety of Boolean algebras.
\begin{thlist}
\item[J5'] $ x = (x' \land y)' \land (x' \land y')' $
\item[A1] $ x \land (y \land z) = (x \land y) \land z$  
\item[A2] $ x \land (y \land z)=x \land (z \land y)$ 
\item[A9] $ x \land x=x. $
\end{thlist}
\end{theorem}

\begin{theorem} The following is a 4-base for the variety of Boolean algebras.
\begin{thlist}
\item[J5']  $ x = (x' \land  y)' \land (x' \land y')'$
\item[A1]  $x \land (y \land z) = (x \land y) \land z$  
\item[A7] $x \land (y \land z)=(y \land z) \land x$   
\item[d] $(x \land y) \land (z \land x)=(x \land (y \land z)) \land x.$
\end{thlist}
\end{theorem}

{\bf Acknowledgements:}  We wish to acknowledge the use of \cite{Mc09} in the early phase of our research that led to the present paper.

\end{document}